\documentclass[12pt,reqno]{amsart}

\usepackage{amsfonts,amsmath,amsthm}
\usepackage{amssymb,epsfig}
\usepackage{enumerate} 

\usepackage{color} 
\definecolor{vert}{rgb}{0,0.6,0}

\usepackage[text={425pt,650pt},centering]{geometry}
\usepackage{graphicx}
\usepackage{epsfig}
\usepackage{tikz}
\usepackage{caption}
\usepackage{color} 
\definecolor{vert}{rgb}{0,0.6,0}

\numberwithin{figure}{section}

\theoremstyle{plain}
\newtheorem{thm}{Theorem}[section]

\newtheorem{defn}{Definition}

\newtheorem{lem}[thm]{Lemma}
\newtheorem{cor}[thm]{Corollary}

\theoremstyle{remark}
\newtheorem{rem}{\bf{Remark}}
\numberwithin{equation}{section}

\newcommand{\R}{\mathbb{R}}

\newcommand{\T}{\mathbb{T}}

\newcommand{\cF}{\mathcal{F}}

\newcommand{\cM}{\mathcal{M}}
\newcommand{\cP}{\mathcal{P}}

\newcommand{\AC}{{\rm AC\,}}

\newcommand{\Li}{L^{\infty}}
\newcommand{\Lip}{{\rm Lip\,}}

\newcommand{\gam}{\gamma}
\newcommand{\del}{\delta}
\newcommand{\ep}{\varepsilon}

\newcommand{\sig}{\sigma}

\newcommand{\Del}{\Delta}

\newcommand{\ol}{\overline}

\newcommand{\tr}{{\rm tr}\,}

\begin{document}

\title[Uniqueness set]
{On uniqueness sets of  additive eigenvalue problems and applications}

\author[H. MITAKE, H. V. TRAN]
{Hiroyoshi Mitake and Hung V. Tran}

\thanks{
The work of HM was partially supported by the JSPS grants: KAKENHI \#15K17574, \#26287024,
\#16H03948, and the work of HT was partially supported by NSF grants DMS-1615944 and DMS-1664424.
}

\address[H. Mitake]{
Institute for Sustainable Sciences and Development, 
Hiroshima University 1-4-1 Kagamiyama, Higashi-Hiroshima-shi 739-8527, Japan}
\email{hiroyoshi-mitake@hiroshima-u.ac.jp}

\address[H. V. Tran]
{
Department of Mathematics, 
University of Wisconsin Madison, 480 Lincoln  Drive, Madison, WI 53706, USA}
\email{hung@math.wisc.edu}


\maketitle

\begin{abstract}
In this paper, we provide a simple way to find uniqueness sets
for additive eigenvalue problems of first and second order Hamilton--Jacobi equations by using a PDE approach. 
An application in finding the limiting profiles for large time behaviors of first order Hamilton--Jacobi equations is also obtained.
\end{abstract}

\section{Introduction}
Let $\T^n$ be the usual $n$-dimensional torus.
Let the Hamiltonian $H = H(x,p) \in C^2(\T^n \times \R^n)$ be such that
\begin{itemize}
\item[(H1)] for every $x\in \T^n$, $p \mapsto H(x,p)$ is convex,
\item[(H2)]  uniformly for  $x \in \T^n$,
\[
\lim_{|p| \to \infty} \frac{H(x,p)}{|p|}=+\infty 
\quad \text{and} \quad
\lim_{|p| \to \infty} \left( \frac{1}{2} H(x,p)^2  + D_x H(x,p)\cdot p \right)= +\infty.
\]
\end{itemize}
The first order additive eigenvalue  (ergodic) problem corresponding to $H$ is
\[
{\rm (E)} \qquad H(x,Dw) = c \qquad \text{ in } \T^n.
\]
Here, $(w,c) \in C(\T^n) \times \R$ is a pair of unknowns.
It was shown in \cite{LPV} that there exists a unique constant $c\in \R$ such that (E) has a viscosity solution $w\in C(\T^n)$.
We denote by $c$ the ergodic constant of (E).
Without loss of generality, we normalize the ergodic constant $c$ to be zero henceforth.

One of the most interesting points to study (E) is that (E) is not monotone in $w$,
and in general, (E) has many viscosity solutions of different types
(see examples in \cite[Chapter 6]{LMT} for instance).
It is therefore fundamental to understand why this \textit{nonuniqueness phenomenon} appears,
and in particular, to find a \textit{uniqueness set} for (E).
It turns out that this has deep relations to Hamiltonian dynamics and weak KAM theory.
In fact, a uniqueness set for (E) has already been studied in \cite{FaB, FS} in the context of weak KAM theory.

In this short paper, we provide a new and simple way to look at this phenomenon by using PDE techniques.
Some applications and generalizations are also provided.

\subsection{Settings and main results}
We first recall the definition of Mather measures. Consider the following minimization problem
\begin{equation}\label{M-min}
\min_{\mu \in \cF} \iint_{\T^n \times \R^n} L(x,v) \,d\mu(x,v),
\end{equation}
where
$L$ is the Legendre transform of $H$, that is, 
\[
L(x,v)=\sup_{ p \in \R^n} \left( p\cdot v - H(x,p) \right) 
\quad\text{for} \ (x,v) \in \T^n \times \R^n,  
\]
and 
\[
\cF=\left\{ \mu \in \mathcal{P}(\T^n \times \R^n)\,:\, \iint_{\T^n \times \R^n} v\cdot D\phi(x) \,d\mu(x,v) = 0 \ \text{ for all } \phi \in C^1(\T^n) \right\}.
\]
Here, $\mathcal{P}(\T^n \times \R^n)$ is the set of all Radon probability measures on $\T^n \times \R^n$.
Measures belong to $\cF$ are called \textit{holonomic measures} associated with (E).

\begin{defn}[Mather measures]
Let $\widetilde \cM\subset\cF$ be the set of all minimizers of \eqref{M-min}.
Each measure in $\widetilde \cM$ is called a Mather measure.
\end{defn}

As we normalize $c=0$, we actually have that (see \cite{M, Man, FaB, FS} for instance)
\begin{equation}\label{min-0}
\min_{\mu \in \cF} \iint_{\T^n \times \R^n} L(x,v) \,d\mu(x,v)=-c=0.
\end{equation}
See \cite{MT6}, \cite[Lemma 6.12]{LMT} for a proof of a more general version this fact.
Here is our first main result.
\begin{thm}\label{thm:uniqueness}
Assume {\rm (H1)--(H2)}.
Let $w_1, w_2$ be two viscosity solutions of ergodic problem {\rm (E)}.
Assume further that 
\begin{equation*}\label{con-unique}
\iint_{\T^n \times \R^n} w_1(x) \,d\mu(x,v) \leq \iint_{\T^n \times \R^n} w_2(x)\,d\mu(x,v) \quad \text{ for all } \mu\in \widetilde \cM.
\end{equation*}
Then $w_1 \leq w_2$ in $\T^n$.
\end{thm}
Let $\cM$ be the projected Mather set on $\T^n$, that is,
\[
\cM = \overline{\bigcup_{\mu \in \widetilde \cM} \text{supp} \left(\text{proj}_{\T^n} \mu\right)}.
\]
Theorem \ref{thm:uniqueness} gives the following straightforward result.
\begin{cor}\label{cor:uniqueness}
Assume {\rm (H1)--(H2)}.
Let $w_1, w_2$ be two  viscosity solutions of ergodic problem {\rm (E)}.
Assume further that $w_1 = w_2$ on $\cM$.
Then $w_1 = w_2$ in $\T^n$.
\end{cor}
Corollary \ref{cor:uniqueness} was derived in \cite{FaB, FS} much earlier.
We provide a simple proof for Theorem \ref{thm:uniqueness} 
in Section \ref{sec:first}, which is a new application of the nonlinear adjoint method introduced in \cite{Ev1} (see also \cite{T1}). 
A generalization of Theorem \ref{thm:uniqueness} to the second order (degenerate viscous) setting,  Theorem \ref{thm:uniqueness2}, is given in Section \ref{sec:degenerate}.
It is worth mentioning that the result of Theorem \ref{thm:uniqueness2} is new in the literature.

\subsection{Application}
We provide here an application in large time behavior. 
In this context, we need to strengthen the convexity of $H$ in (H1).
\begin{itemize}
\item[(H1')] There exists $\gam>0$ such that
\[
D^2_{pp}H(x,p) \geq \gam I_n \quad \text{ for all } (x,p) \in \T^n \times \R^n.
\]
\end{itemize}
Here, $I_n$ is the identity matrix of size $n$. 

Under assumptions {\rm(H1')}, {\rm (H2)} and that the ergodic constant $c=0$, 
for given $u_0 \in \Lip(\T^n)$, the viscosity solution $u \in C(\T^n \times [0,\infty))$ of 
the Cauchy problem
\[
{\rm (C)} \qquad 
\begin{cases}
u_t + H(x,Du)=0 \quad &\text{ in } \T^n \times (0,\infty),\\
u(x,0)= u_0(x) \quad &\text{ on } \T^n.
\end{cases}
\]
has the following large time behavior 
\begin{equation}\label{thm:large-time}
\lim_{t \to \infty} \|u(\cdot,t) - v\|_{L^\infty(\T^n)} =0, 
\end{equation}
where $v \in \Lip(\T^n)$ is a viscosity solution of {\rm (E)}. 
This result was first proved in \cite{F1}. 
Notice that there are various different ways to prove it (see \cite{BS, CGMT, LMT} and the references therein).
We say that $v$ is the \textit{asymptotic profile} of $u$, and denote it by $u^\infty$, 
or $u^\infty[u_0]$ to display the clear dependence on the initial data $u_0$.

We now give a representation formula for $u^\infty[u_0]$.
\begin{thm}[Asymptotic profiles] \label{thm:profile}
Assume that {\rm(H1')} and {\rm (H2)} hold, and the ergodic constant $c=0$.
For given $u_0 \in \Lip(\T^n)$, let $u^\infty[u_0]$  be the corresponding asymptotic profile.
Then, we have 
\begin{itemize}
\item[(i)] $u^\infty[u_0](y) = u_0^-(y)$ for all $y \in \cM$, 
\item[(ii)] $u^\infty[u_0](x) = \min \left\{ d(x,y) + u_0^-(y)\,:\, y \in \cM \right\}$ for all $x\in \T^n$.
\end{itemize}
Here, 
\begin{align*}
&u_0^-(x) 
= \sup \left\{ v(x)\,:\, v \leq u_0 \ \text{on} \ \T^n, \ \text{and $v$ is a subsolution to {\rm(E)}} \right\}, \\
&
d(x,y) = \sup\left\{v(x)-v(y)\,:\, v \text{ is a subsolution to {\rm(E)}}\right\}.
\end{align*}
\end{thm}
Theorem \ref{thm:profile} was first proved in \cite[Theorem 3.1]{DS}, and our purpose is to give a different proof in Section \ref{sec:app}, 
which seems to be simpler.


\section{Uniqueness set of the ergodic problem} \label{sec:first}
We present in this section the proof of Theorem \ref{thm:uniqueness}.

\begin{proof}[Proof of Theorem {\rm\ref{thm:uniqueness}}]
We use ideas introduced in \cite{CGMT}.

For each $i=1,2$ and each $\ep>0$,
let $u_i^\ep$ be the viscosity solution to the Cauchy problem
\begin{equation} \label{C-ep}
\begin{cases}
\ep(u^{\ep}_i)_t + H(x,Du^{\ep}_i) = \ep^4\Del u^{\ep}_i \qquad &\text { in } \T^n \times (0,1),\\
u^{\ep}_i(x,0) = w_i(x) \qquad &\text{ on } \T^n.
\end{cases}
\end{equation}
Without the viscosity term, \eqref{C-ep} becomes
\begin{equation} \label{C-0}
\begin{cases}
\ep(u_i)_t + H(x,Du_i) = 0 \qquad &\text { in } \T^n \times (0,1),\\
u_i(x,0) = w_i(x) \qquad &\text{ on } \T^n.
\end{cases}
\end{equation}
It is clear that the unique viscosity solution to \eqref{C-0} is $u_i(x,t) = w_i(x)$ for all $(x,t) \in \T^n \times [0,1)$
because of the fact that $w_i$ is a viscosity solution to (E).
Thanks to (H2), by a standard argument, 
there exists $C>0$ independent of $\ep$ such that
\begin{equation}\label{grad-bound0}
\|Du^\ep_i\|_{\Li(\T^n \times (0,1))} \leq C
\end{equation}
and
\begin{equation} \label{eqn-error0}
\|u_i^{\ep}- w_i\|_{\Li(\T^n\times(0,1))}\le C\ep.
\end{equation}
See \cite[Propositions 4.15 and 5.5]{LMT} for the proofs of similar versions of \eqref{grad-bound0} and \eqref{eqn-error0} for instance. 
Our plan is to use $u_1^\ep, u_2^\ep$ to deduce the conclusion as $\ep \to 0$.

For any $x_0 \in \T^n$, let $\sig^\ep$ be the solution to 
\begin{equation*}
\begin{cases}
-\ep \sig^\ep_t -\text{div}(D_p H(x,Du^\ep_2) \sig^\ep) = \ep^4\Del \sig^{\ep} \qquad &\text { in } \T^n \times (0,1),\\
\sig^\ep(x,1)= \del_{x_0} \qquad &\text{ on } \T^n.
\end{cases}
\end{equation*}
Here $\del_{x_0}$ is the Dirac delta mass at $x_0$.

By convexity of $H$ in (H1), we have
\[
\ep(u_1^\ep - u_2^\ep)_t + D_pH(x,Du_2^\ep)\cdot D(u_1^\ep-u_2^\ep) \leq \ep^4\Del (u_1^\ep-u_2^\ep).
\]
Multiply this by $\sig^\ep$, integrate on $\T^n$, and note that 
\begin{align*}
&
\int_{\T^n} 
\left(-D_pH(x,Du_2^\ep)\cdot D(u_1^\ep-u_2^\ep)+\ep^4\Del (u_1^\ep-u_2^\ep)\right)\sig^\ep\,dx\\
=&\, 
\int_{\T^n} 
\left(\text{div}(D_p H(x,Du^\ep_2) \sig^\ep)+\ep^4\Del \sig^{\ep} \right)(u_1^\ep-u_2^\ep)\,dx
=-\int_{\T^n} 
\ep\sig_{t}^{\ep}(u_1^\ep-u_2^\ep)\,dx. 
\end{align*}
Thus, 
\[
\frac{d}{dt} \int_{\T^n} (u_1^\ep - u_2^\ep) \sig^\ep \,dx \leq 0,
\]
which yields
\begin{equation}\label{u-ineq0}
(u_1^\ep - u_2^\ep)(x_0,1) \leq \int_0^1 \int_{\T^n} (u_1^\ep - u_2^\ep) \sig^\ep \,dxdt.
\end{equation}

In light of the Riesz theorem, there exists $\nu^{\ep}\in\cP(\T^n\times\R^n)$ such that 
\begin{equation}\label{def-mu0}
\iint_{\T^n\times\R^n}\varphi(x,p)\,d\nu^{\ep}(x,p)
=
\int_{0}^{1}\int_{\T^n}\varphi(x,Du_2^\ep)\sig^{\ep}\,dxdt\quad
\text{for all} \ \varphi\in C_{c}(\T^n\times\R^n). 
\end{equation}
Then,  \eqref{u-ineq0} becomes
\begin{equation}\label{ineq-1-0}
(u_1^\ep - u_2^\ep)(x_0,1) \leq \iint_{\T^n\times\R^n} (u_1^\ep - u_2^\ep) \,d\nu^{\ep}(x,p). 
\end{equation}

Thanks to \eqref{grad-bound0}, we have that $\text{supp}(\nu^\ep) \subset \T^n \times \ol{B}(0,C)$.
There exists $\{\ep_{j}\}\to 0$ such that $\nu^{\ep_j}\rightharpoonup \nu\in\cP(\T^n\times\R^n)$ 
as $j\to\infty$ weakly in the sense of measures. 
We set $\mu \in \cP(\T^n \times \R^n)$ be such that
\begin{equation}\label{def-mu-nu-0}
\iint_{\T^n \times \R^n} \varphi(x,p)\, d\nu(x,p) = \iint_{\T^n \times \R^n} \varphi(x,D_v L(x,v))\,d\mu(x,v).
\end{equation}
We provide a proof that $\mu$ is a Mather measure in Lemma \ref{lem:mu} below for completeness 
(see also \cite[Proposition 2.3]{MT6}, \cite[Proposition 6.11]{LMT}).

Sending $j \to \infty$ in \eqref{ineq-1-0} and using \eqref{eqn-error0} to yield
\[
w_1(x_0) -w_2(x_0) \leq    \iint_{\T^n\times\R^n} (w_1  - w_2)\,d\mu(x,v) \leq 0.
\qedhere
\]
\end{proof}

\begin{lem}\label{lem:mu}
For each $\ep>0$, let $\nu^\ep$ be the measure defined in \eqref{def-mu0}.
Assume that there exists a sequence $\{\ep_j\} \to 0$ such that
$\nu^{\ep_j} \rightharpoonup \nu \in \cP(\T^n \times \R^n)$ as $j \to \infty$ weakly in the sense of measures.
Let $\mu$ be a measure defined through $\nu$ by \eqref{def-mu-nu-0}.
Then $\mu$ is a Mather measure.
\end{lem}

\begin{proof}
Fix any $\phi \in C^1(\T^n)$, 
and consider a family  $\left\{\phi^m \right\} \subset C^\infty(\T^n)$ such that $\phi^m\to\phi$ in $C^1(\T^n)$ as $m \to \infty$.
 
Multiply the adjoint equation with $\phi^m$ and integrate on $\T^n \times [0,1]$ to imply
\begin{multline*}
\ep \int_{\T^n} \phi^m(x) \sig^\ep(x,0)\,dx - \ep \phi^m(x_0) +\int_0^1 \int_{\T^n} D_p H(x,Du_2^\ep) \cdot D\phi^m(x) \sig^\ep(x,t)\,dxdt \\
= \ep^4 \int_0^1 \int_{\T^n} \Del \phi^m(x) \sig^\ep(x,t)\,dxdt.
\end{multline*}
Let $\ep =\ep_j \to 0$ and $m\to\infty$ in this order to get
\[
\iint_{\T^n \times \R^n} D_p H(x,p)\cdot D\phi(x) \, d\nu(x,p) = 
\iint_{\T^n \times \R^n} v \cdot D\phi(x) \, d\mu(x,v) = 0. 
\]
Thus, $\mu \in \cF$.

We rewrite \eqref{C-ep} as
\[
\ep (u_2^\ep)_t + D_p H(x,Du_2^\ep)\cdot Du_2^\ep - \ep^4 \Del u_2^\ep = D_p H(x,Du_2^\ep)\cdot Du_2^\ep - H(x,Du_2^\ep).
\]
Multiply this by $\sig^\ep$ and integrate on $\T^n \times [0,1]$ to yield
\[
\ep u_2^\ep(x_0,1) - \ep \int_{\T^n} u_2^\ep(x,0) \sig^\ep(x,0)\,dx
=\int_0^1 \int_{\T^n}  (D_p H(x,Du_2^\ep)\cdot Du_2^\ep - H(x,Du_2^\ep))\sig^\ep\,dxdt.
\]
We again let $\ep=\ep_j \to 0$ to achieve that
\[
0=\iint_{\T^n \times \R^n} (D_p H(x,p)\cdot p - H(x,p))\, d\nu(x,p)
=\iint_{\T^n\times \R^n} L(x,v)\, d\mu(x,v).
\]
Also, note that we have 
\begin{equation}\label{ge-0}
\iint_{\T^n\times\R^n}L(x,v)\,d\mu\ge0
\qquad\text{for all} \ \mu\in\cF, 
\end{equation}
which, together with \eqref{min-0}, completes the proof. 
See \cite[Lemma 6.12]{LMT} for a proof of \eqref{ge-0}.
\end{proof}


\section{Application}\label{sec:app}
In this section, we always assume that (H1')--(H2) hold and that the ergodic constant $c=0$.
\begin{lem} \label{lem:sub}
Assume that $u_0$ is a viscosity subsolution of {\rm (E)}.
Then,  
\[
u^\infty[u_0] = u_0  \quad \text{ on } \cM.
\]
\end{lem}

\begin{proof}
We write $u^\infty$ for $u^\infty[u_0]$ in the proof for simplicity.

By the usual comparison principle, we have $u(x,t) \geq u_0(x)$ for all $(x,t) \in \T^n \times [0,\infty)$.
Hence, $u^\infty \geq u_0$ on $\T^n$.

Next, let $\rho$ be a standard mollifier in $\R^n$. For each $\del>0$, let $\rho^\del(x) = \del^{-n} \rho(\del^{-1}x)$ for all $x\in \R^n$.
Let $u^\del = \rho^\del *u$. Then due to the convexity of $H$ in $p$, 
$u^\del$ is a subsolution to
\[
u^\del_t + H(x,Du^\del) \leq C \del \quad \text{ in }\T^n \times (0,\infty).
\]
For any Mather measure $\mu\in\widetilde{\cM}$, by the holonomic and minimizing properties, we have
\begin{align*}
\frac{d}{dt} \iint_{\T^n \times \R^n} u^\del(x,t) \, d\mu 
&= \iint_{\T^n\times \R^n} (u^\del_t + v \cdot Du^\del  - L(x,v))\, d\mu\\
& \le \iint_{\T^n\times \R^n} u^\del_t + H(x,Du^\del)\, d\mu
\leq C\del.
\end{align*}
Therefore, for any $T>0$,
\[
\iint_{\T^n\times \R^n} u^\del(x,T)\,d\mu 
\leq \iint_{\T^n\times \R^n} (u_0)^\del(x)\,d\mu + C\del T.
\]
Let $\del \to 0$ and $T \to \infty$ in this order to yield
\[
\iint_{\T^n\times \R^n} u^\infty \,d\mu \leq \iint_{\T^n\times \R^n} u_0\,d\mu. 
\]
Combined with $u^\infty \geq u_0$ on $\T^n$, we obtain 
$u^\infty= u_0$ on $\cM$,  
which completes the proof.
\end{proof}

\begin{rem}
Notice that we get 
\[
u(x,t) = u_0(x) \quad \text{ for all } x\in \cM,\ t \in [0,\infty),
\]
in the above proof.
\end{rem}

We present next the proof of Theorem \ref{thm:profile}.
Before proceeding to the proof, it is important noticing that $d$ has the following representation formula
\[
d(x,y) =\inf \left\{ \int_0^t L(\gam(s),-\dot \gam(s))\,ds\,:\, t>0, \gam \in \AC([0,t],\T^n), \gam(0)=x, \gam(t)=y\right\}.
\]

\begin{proof}[Proof of Theorem {\rm\ref{thm:profile}}]
It is enough to give only the proof of (i).
The second claim (ii) follows immediately from Corollary \ref{cor:uniqueness}, claim (i) and the representation formulas of $d$ as well as of solutions to (E).

\smallskip

By the definition of $u_0^-$, we have $u_0^-\le u_0$ on $\T^n$. 
In light of the comparison principle, $u_0^-\le u$ on $\T^n\times[0,\infty)$, which implies 
$u_0^- \leq u^\infty$ on $\T^n$. 

We prove the reverse inequality holds on $\cM$.
Fix $y\in \cM$ and $z\in \T^n$. Set $w_0^z(x) = u_0(z) + d(x,z)$ for $x\in \T^n$. 
Then, note that $w_0^z$ is a viscosity subsolution to (E). 
Let $w$ be the solution to (C) with initial data $w_0^z$.
Thanks to Lemma \ref{lem:sub}, we get 
\begin{equation}\label{eq:w1}
w(y,t) = w_0^z(y) = u_0(z) + d(y,z) \quad \text{ for all } t\in  [0,\infty).
\end{equation}
For a large $t >1$, pick $\gam:[0,t] \to \T^n$ to be an optimal path such that $\gam(0)=y$ and
\[
w(y,t) = w_0^z(\gam(t)) + \int_0^t L(\gam(s),-\dot \gam(s))\,ds= u_0(z) + d(\gam(t),z)  + \int_0^t L(\gam(s),-\dot \gam(s))\,ds.
\]

On the other hand, for any $\ep>0$, there exists $t_\ep>0$ and a path $\gam: [t,t+t_\ep] \to \T^n$ with $\gam(t+t_\ep)=z$ satisfying
\[
d(\gam(t),z) \geq \int_t^{t+t_\ep} L(\gam(s),-\dot \gam(s))\,ds - \ep.
\]
Combine the two relations above to imply
\begin{equation}\label{eq:w2}
w_0^z(y) +\ep \geq u_0(z) +\int_0^{t+t_\ep} L(\gam(s),-\dot \gam(s))\,ds \geq u(y,t+t_\ep).
\end{equation}
By letting $t \to \infty$ in \eqref{eq:w2}, one gets
\[
w_0^z(y) + \ep \geq u^\infty(y).
\]
Next, let $\ep \to 0$ to conclude that $u_0(z) + d(y,z) \geq u^\infty(y)$. Vary $z$ to yield
\[
u^\infty(y) \leq \min_{z\in \T^n} (u_0(z) + d(y,z)).
\]
Notice here that in view of the inf-stability of viscosity subsolutions to 
convex first order Hamilton--Jacobi equations, we have 
$\min_{z\in \T^n} (u_0(z) + d(y,z))=u_0^-(y)$, which finishes the proof. 
\end{proof}


\section{Generalization: degenerate viscous cases} \label{sec:degenerate}

In this section, we present a generalization of Theorem \ref{thm:uniqueness} to the second order (degenerate viscous) setting. In this setting, the ergodic problem is
\[
{\rm (VE)} \qquad 
H(x,Dw) = \tr\left(A(x)D^2 w\right) + c \quad \text{ in } \T^n.
\]
As above, $(w,c) \in C(\T^n) \times \R$ is a pair of unknowns.
Here $A:\T^n \to \mathbb M^{n \times n}_{\text{sym}}$ is the diffusion matrix,
where $ \mathbb M^{n \times n}_{\text{sym}}$ is the set of all $n\times n$ real symmetric matrices.
We need the following assumptions. 
\begin{itemize}
\item[(H2')] There exist $\gam>1$ and $C>0$ such that, for all $(x,p) \in \T^n \times \R^n$,
\[
\begin{cases}
\displaystyle \frac{1}{C}|p|^\gam - C \leq H(x,p) \leq C(|p|^\gam+1),\\
|D_x H(x,p)| \leq C(1+|p|^\gam),\\
|D_p H(x,p)| \leq C(1 + |p|^{\gam-1}).
\end{cases}
\]
\item[(H3)] 
$A(x)=(a_{ij}(x))_{1\leq i,j \leq n} \in \mathbb M^{n \times n}_{\text{sym}}$ with $A\ge0$, and  $a_{ij}\in C^2(\T^n)$ for all $1\leq i,j \leq n$.  
\end{itemize}
By normalization, we always assume that $c=0$ in this section.
In fact, under assumptions (H1), (H2') and (H3), for any $w\in C(\T^n)$ solving (VE), $w\in \Lip(\T^n)$ (see \cite[Theorem 3.1]{AT}).

\begin{defn} \label{def:generalized}
Let $\widetilde \cM_V$ be the set of all minimizers of the minimizing problem
\begin{equation}\label{Mv-min}
\min_{\mu \in \cF} \iint_{\T^n \times \R^n} L(x,v) \,d\mu(x,v),
\end{equation}
where
\[
\cF_V=\left\{ \mu \in \mathcal{P}(\T^n \times \R^n)\,:\, 
\iint_{\T^n \times \R^n} v\cdot D\phi -a_{ij}\phi_{x_ix_j}\,d\mu(x,v) = 0 \ 
\text{ for all } \phi \in C^2(\T^n) \right\}.
\]
Each measure in $\widetilde \cM_V$ is called a generalized Mather measure.
\end{defn}

Because of normalization that $c=0$, as in the first order case, one has that
\begin{equation}\label{min-viscous}
\min_{\mu \in \cF_V} \iint_{\T^n \times \R^n} L(x,v) \,d\mu(x,v)=0.
\end{equation}
The proof of this claim follows  \cite[Lemma 6.12]{LMT}. To be more precise,  \cite[Lemma 6.12]{LMT} deals with the special case $A(x) = a(x) I_n$ where $a\in C^2(\T^n,[0,\infty))$
and $I_n$ is the identity matrix of size $n$. 
For general diffusion matrix $A$ satisfying (H3), we perform first inf-sup convolutions, and then normal convolution of a solution $w$ of (VE).
See also \cite{IMT1} for a form of \eqref{min-viscous} in fully nonlinear, degenerate elliptic PDE settings.

\begin{thm}\label{thm:uniqueness2}
Assume {\rm (H1), (H2'), (H3)}.
Let $w_1, w_2$ be two continuous viscosity solutions of ergodic problem {\rm (E)}.
Assume further that 
\begin{equation*}\label{con-unique}
\iint_{\T^n \times \R^n} w_1(x) \,d\mu(x,v) \leq \iint_{\T^n \times \R^n} w_2(x)\,d\mu(x,v) \quad \text{ for all } \mu\in \widetilde \cM_V.
\end{equation*}
Then $w_1 \leq w_2$ in $\T^n$.
\end{thm}

\begin{proof}
We basically repeat the proof of Theorem \ref{thm:uniqueness}.

For each $k=1,2$ and each $\ep>0$,
let $u_k^\ep$ be the solution to the Cauchy problem
\begin{equation*} \label{VC-ep}
\begin{cases}
\ep(u^{\ep}_k)_t + H(x,Du^{\ep}_k) = a_{ij}(u_k^\ep)_{x_ix_j}+\ep^4\Del u^{\ep}_k \qquad &\text { in } \T^n \times (0,1),\\
u^{\ep}_k(x,0) = w_k(x) \qquad &\text{ on } \T^n.
\end{cases}
\end{equation*}
Without the viscosity $\ep^4 \Del u^{\ep}_k $, \eqref{C-ep} becomes
\begin{equation} \label{VC-0}
\begin{cases}
\ep(u_k)_t + H(x,Du_k) = a_{ij}(u_k)_{x_ix_j} \qquad &\text { in } \T^n \times (0,1),\\
u_k(x,0) = w_k(x) \qquad &\text{ on } \T^n,
\end{cases}
\end{equation}
It is clear that the unique viscosity solution to \eqref{VC-0} is $u_k(x,t) = w_k(x)$ for all $(x,t) \in \T^n \times [0,1)$
because of the fact that $w_k$ is a solution to (VE).
Thanks to (H2') (see \cite[Theorem 4.5]{LMT} for instance),
there exists $C>0$ independent of $\ep$ such that
\begin{equation}\label{grad-bound}
\|Du^\ep_i\|_{\Li(\T^n \times (0,1))} \leq C \quad \text{and} \quad
\|u_i^{\ep}- w_i\|_{\Li(\T^n\times(0,1))}\le C\ep.
\end{equation}
As above, we use $u_1^\ep, u_2^\ep$ to deduce the conclusion as $\ep \to 0$.

For any $x_0 \in \T^n$, let $\sig^\ep$ be the solution to 
\begin{equation*}
\begin{cases}
-\ep \sig^\ep_t -\text{div}(D_p H(x,Du^\ep_2) \sig^\ep) = 
(a_{ij}\sig^{\ep})_{x_ix_j}+\ep^4\Del \sig^{\ep} \qquad &\text { in } \T^n \times (0,1),\\
\sig^\ep(x,1)= \del_{x_0} \qquad &\text{ on } \T^n.
\end{cases}
\end{equation*}
Here $\del_{x_0}$ is the Dirac delta mass at $x_0$.

By convexity of $H$, we have
\[
\ep(u_1^\ep - u_2^\ep)_t + D_pH(x,Du_2^\ep)\cdot D(u_1^\ep-u_2^\ep) \leq 
a_{ij}(u_1^\ep-u_2^\ep)_{x_ix_j}+\ep^4\Del (u_1^\ep-u_2^\ep).
\]
Multiply this by $\sig^\ep$ and integrate on $\T^n$ to yield
\[
\frac{d}{dt} \int_{\T^n} (u_1^\ep - u_2^\ep) \sig^\ep \,dx \leq 0.
\]
Hence,
\begin{equation}\label{u-ineq}
(u_1^\ep - u_2^\ep)(x_0,1) \leq \int_0^1 \int_{\T^n} (u_1^\ep - u_2^\ep) \sig^\ep \,dxdt.
\end{equation}

Let $\nu^{\ep}\in\cP(\T^n\times\R^n)$ be the measure satisfying 
\begin{equation*}\label{def-mu}
\iint_{\T^n\times\R^n}\varphi(x,p)\,d\nu^{\ep}(x,p)
=
\int_{0}^{1}\int_{\T^n}\varphi(x,Du_2^\ep)\sig^{\ep}\,dxdt\quad
\text{for all} \ \varphi\in C_{c}(\T^n\times\R^n). 
\end{equation*}
Then,  \eqref{u-ineq} becomes
\begin{equation}\label{ineq-1}
(u_1^\ep - u_2^\ep)(x_0,1) \leq \iint_{\T^n\times\R^n} (u_1^\ep - u_2^\ep) \,d\nu^{\ep}(x,p). 
\end{equation}

Thanks to \eqref{grad-bound}, we have that $\text{supp}(\nu^\ep) \subset \T^n \times \ol{B}(0,C)$.
There exists $\{\ep_{j}\}\to 0$ such that $\nu^{\ep_j}\rightharpoonup \nu\in\cP(\T^n\times\R^n)$ 
as $j\to\infty$ weakly in the sense of measures. 
We set $\mu \in \cP(\T^n \times \R^n)$ be such that
\begin{equation*}\label{def-mu-nu}
\iint_{\T^n \times \R^n} \varphi(x,p) d\nu(x,p) = \iint_{\T^n \times \R^n} \varphi(x,D_v L(x,v))\,d\mu(x,v).
\end{equation*}
Note that $\mu$ is a generalized Mather measure defined in Definition \ref{def:generalized}.  
We refer to \cite[Proposition 2.3]{MT6} or \cite[Proposition 6.11]{LMT} for the details.

Sending $j \to \infty$ in \eqref{ineq-1} and using \eqref{grad-bound} to yield
\[
w_1(x_0) -w_2(x_0) \leq    \iint_{\T^n\times\R^n} (w_1  - w_2)\,d\mu(x,v) \leq 0.
\qedhere
\]
\end{proof}

Let $\cM_V$ be the generalized projected Mather set on $\T^n$, that is,
\[
\cM_V = \overline{\bigcup_{\mu \in \widetilde \cM_V} \text{supp} \left(\text{proj}_{\T^n} \mu\right)}.
\]
Theorem \ref{thm:uniqueness2} gives the following straightforward result.
\begin{cor}
Assume {\rm (H1), (H2'), (H3)}.
Let $w_1, w_2$ be two continuous viscosity solutions of ergodic problem {\rm (VE)}.
Assume further that $w_1 \leq w_2$ on $\cM_V$.
Then $w_1 \leq w_2$ in $\T^n$.
\end{cor}

\end{document}